\documentclass{amsart}

\usepackage{amssymb,amsmath,color}
\usepackage{amsthm}
\usepackage{url}
\usepackage{tikz}
\usepackage[enableskew]{youngtab}
\usepackage{ytableau,varwidth
}

\definecolor{red}{rgb}{1,0,0}
\definecolor{blue}{rgb}{.2,.2,.8}

\def\m{\mu}

\def\T{\mathcal T_r}

\def\P{\mathcal P}

\def\D{\mathcal D}
\def\bD{\mathfrak D}
\def\bO{{\mathfrak O}}

\def\N{\mathbb N}

\def\del{\partial}

\newtheorem{theorem}{Theorem}

\newcommand{\ds}{\displaystyle}

\makeatletter
\newcommand{\tpmod}[1]{{\@displayfalse\pmod{#1}}}
\makeatother

\begin{document}

\title[]{Beck-type companion identities for Franklin's identity via a modular refinement}

\author{Cristina Ballantine}\address{Department of Mathematics and Computer Science\\ College of the Holy Cross \\ Worcester, MA 01610, USA \\} 
\email{cballant@holycross.edu} 
\author{Amanda Welch} \address{Department of Mathematics and Computer Science\\ College of the Holy Cross \\ Worcester, MA 01610, USA \\} \email{awelch@holycross.edu}

\begin{abstract}The original Beck conjecture, now a theorem due to Andrews, states that the difference in the number of parts in all partitions into odd parts and the number of parts in all strict partitions is equal to the number of partitions whose set of even parts has one element, and also to the number of partitions with exactly one part repeated. This is a companion identity to Euler's identity. The theorem has been generalized by Yang to a companion identity to Glaisher's identity. Franklin generalized Glaisher's identity, and in this article, we provide a Beck-type companion identity for Franklin's identity and prove it via a modular refinement. We provide both analytical and combinatorial proofs.  Andrews' and Yang's respective theorems fit naturally into this very general frame. We also give a generalization to Franklin's identity of the second Beck-type companion identity proved by Andrews and Yang in their respective work. \end{abstract}

\maketitle



\section{Introduction} \label{intro}

Partition identities are often statements asserting that the set $\P_X$ of partitions of $n$ subject to condition $X$ is equinumerous to the set $\P_Y$ of partitions of $n$ subject to condition $Y$. A Beck-type identity is a companion  identity to $|\P_X|=|\P_Y|$ asserting that the difference  between the number of parts in all partitions in $\P_X$ and the number of parts in all partitions in $\P_Y$ is related to  $|\P_{X'}|$ and also to $|\P_{Y'}|$, where $X'$, respectively $Y'$, is a condition on partitions that is a very slight relaxation of  condition $X$, respectively  $Y$. These types of companion identities were first noticed in  two conjectures by George Beck which appeared in The On-Line Encyclopedia of Integer Sequences  on the pages for sequences A090867 \cite{B1} and A265251 \cite{B2}. They were companion identities to Euler's identity. We will explain his conjectures and the history of their proofs after introducing some notation. 

Let $n$ be a non-negative integer.  A \textit{partition} $\lambda$ of $n$ is a non-increasing sequence of positive integers $\lambda=(\lambda_1, \lambda_2, \ldots, \lambda_\ell)$ that add up to $n$, i.e., $\ds\sum_{i=1}^\ell\lambda_i=n$. The numbers $\lambda_i$ are called the \textit{parts} of $\lambda$ and $n$ is called the \textit{size} of $\lambda$.  The number of parts of the partition is called the \textit{length} of $\lambda$ and is denoted by $\ell(\lambda)$. 

We will also use the exponential notation for parts in a partition. The exponent of a part is the multiplicity of the part in the partition. For example, $(5^2, 4, 3^3, 1^2)$ denotes the partition $(5,5, 4, 3, 3, 3, 1, 1)$. Mostly, we will use the exponential notation when referring to rectangular partitions, i.e., partitions in which all parts are equal. Thus, we write $(m^i)$ for the partition consisting of $i$ parts equal to $m$. 

%The Ferrers diagram of a partition $\l=(\l_1, \l_2, \ldots, \l_{\ell})$ is an array of left justified boxes such that the $i$th row from to top contains $\l_i$ boxes. For example, the Ferrers diagram of the partition $(5,5, 3, 3, 2,1)$ is shown below. $$\tiny\ydiagram{5, 5, 3, 3,2, 1}$$

Next, we define two  operations on partitions. Let  $\lambda=(\lambda_1, \lambda_2, \ldots, \lambda_{\ell(\lambda)})$ and $\m=(\m_1, \m_2, \ldots, \m_{\ell(\m)})$ be two partitions. %we define partitions $\l\cup \m$ and $\l\setminus \m$.
We denote by $\lambda\cup \mu$   the partition whose  parts are precisely the parts of $\lambda$ and $\m$, i.e., $\lambda_1, \lambda_2, \ldots, \lambda_{\ell(\lambda)}, \m_1, \m_2, \ldots, \m_{\ell(\m)}$, arranged in non-increasing order. 
The partition $\lambda \setminus \m$ is defined only if all parts of $\mu$ (considered with multiplicity) are also parts of $\lambda$. Then, $\lambda \setminus \m$ is the partition obtained from $\lambda$ by removing all parts of $\mu$ (with multiplicity). 

%The partition $\l+\mu$ is  the partition with parts $\l_1+\m_1,\l_2+\m_2, \ldots, \l_k+\m_k$, where $k=\max(\ell(\l), \ell(\m))$ and, if $\ell(\l)<k$ or $\ell(\m)<k$, the respective partition is padded with parts equal to $0$. 

%If $\ell(\m)\leq \ell(\l)$ and $\m_i\leq \l_i$ for all $1\leq i\leq \ell(\l)$, we define the partition $\l-\mu$ as  the partition $(\l_1-\m_1,\l_2-\m_2, \ldots, \l_k-\m_k)$, where, if $\ell(\m)<\ell(\l)$, the  partition $\mu$ is padded with parts equal to $0$, i.e. $\m_{\ell(\mu)+1}=\cdots =\m_{\ell(\l)}=0$. 

%For a  non-negative integer $n$,  a \textit{composition} $\a$ of $n$ is a sequence of positive integers $\a=(\a_1, \a_2, \ldots, \a_k)$ that add up to $n$. Thus $(3,2, 3, 1)$ and $(3, 1, 3, 2)$ are different compositions of $9$. The sum and difference of compositions  are defined analogous to the sum and difference of partitions. 

Throughout the article, we make use of the following notation.

We denote by $\mathcal{O}_{j,r}(n)$, respectively $\mathcal{O}_{\leq j, r}(n)$,  the set of   partitions of $n$ with  exactly $j$, respectively at most $j$, different parts (possibly repeated) divisible by  $r$ and by  
$\D_{j,r}(n)$, respectively $\D_{\leq j, r}(n)$, the set of partitions of $n$ in which exactly $j$, respectively at most $j$, different parts are repeated at least $r$ times and all other parts appear no more than $r-1$ times.

Euler's partition identity states that for all non-negative integers $n$,  $$|\mathcal O_{0,2}(n)|=|\mathcal D_{0,2}(n)|.$$ 

Glaisher's identity generalizes Euler's identity and states that for all non-negative integers $n$ and all integers $r\geq 2$,  \begin{equation}\label{gl}|\mathcal O_{0,r}(n)|=|\mathcal D_{0,r}(n)|.\end{equation}

In 1883, Franklin \cite{F83} proved the following generalization of Glaisher's identity. 

\begin{theorem}[Franklin] \label{FT} For all non-negative integers $n, j$ and all integers $r\geq 2$, \begin{equation}\label{FTid}|\mathcal O_{j,r}(n)|=|\mathcal D_{j,r}(n)|.\end{equation} 
\end{theorem}

 George Beck conjectured a companion identity to Euler's partition identity, namely \begin{equation}\label{beck1}|\mathcal O_{1,2}(n)|=|\mathcal D_{1,2}(n)|=b(n),\end{equation} where $b(n)$ is the difference between the number of parts in all partitions in $\mathcal O_{0,2}(n)$ and the number of parts in all partitions in $\mathcal D_{0,2}(n)$.  Andrews proved these identities in \cite{A17}  using generating functions.  Since then,  in a fairly short time, many articles appeared giving generalizations of this result as well as combinatorial proofs in many cases. See for example \cite{FT17, Y18, BB19, LW19, LW19b, LW19c, AB19, BW20a, BW20b}. Some authors have started referring to these companion identities as Beck-type identities. Some of the earlier generalizations \cite{Y18} gave companion identities to Glaisher's identity \eqref{gl}. Let $b_{j,r}(n)$ denote  the difference between the number of parts in all partitions in $\mathcal{O}_{j,r}(n)$ and the number of parts in all partitions in $\D_{j,r}(n)$. %, i.e., $$b_{j,r}(n)=\sum_{\l\in \O_{j,r}(n)} \ell(\l)-\sum_{\l\in \D_{j,r}(n)} \ell(\l).$$ 
 The Beck-type identity accompanying \eqref{gl} is \begin{equation}\label{bg1} |\mathcal{O}_{1,r}(n)|=|\D_{1,r}(n)|=\frac{1}{r-1}b_{0,r}(n).\end{equation}  

The main result of this article gives a Beck-type companion identity for Franklin's identity \eqref{FTid}. Let $b_{\leq j,r}(n)$ denote  the difference between the number of parts in all partitions in $\mathcal{O}_{\leq j,r}(n)$ and the number of parts in all partitions in $\D_{\leq j,r}(n)$, i.e., $$b_{\leq j,r}(n)=\sum_{\lambda\in \mathcal{O}_{\leq j,r}(n)} \ell(\lambda)-\sum_{\lambda\in \D_{\leq j,r}(n)} \ell(\lambda).$$

\begin{theorem} \label{main1} Let  $n,j, r$ be non-negative integers with  $r\geq 2$. Then, $$\frac{1}{r-1}b_{\leq j,r}(n)=(j+1)|\mathcal{O}_{j+1,r}(n)|= (j+1)|\D_{j+1,r}(n)|.$$
\end{theorem}
The case $j=0$ gives \eqref{bg1}. Theorem \ref{main1} is obtained from the repeated application of the next theorem for which  both analytic and combinatorial proofs are given in \cite{BW20c} (which was written as an extended abstract for this article). 

\begin{theorem}\label{main} For all non-negative integers $n,j$ and all integers  $r\geq 2$, we have \begin{equation}\label{eqmain} \frac{1}{r-1}b_{j,r}(n) =(j+1)|\mathcal{O}_{j+1,r}(n)|-j|\mathcal{O}_{j,r}(n)|= (j+1)|\D_{j+1,r}(n)|-j|\D_{j,r}(n)|.\end{equation}
\end{theorem} 

 Note that Theorem \ref{main} itself can be viewed as a generalization of \eqref{bg1}. However, for $j\geq 1$, the right hand side of \eqref{eqmain} is non-positive. %We remark that while \cite{LW19b} gives a another generalization of \eqref{bg1} involving the number of parts in $\O_{j,r}(n)$ (but not in $\D_{j,r}(n)$), their result does not lead to a natural generalization as in Theorem \ref{main1} nor to the further generalization to Euler pairs described in Section \ref{fg}. 

 In this article we give a new proof of Theorem \ref{main} via a modular refinement. First, we introduce the following notation. Let $r$ be a positive integer such that $r \geq 2$ and let $0\leq t\leq r-1$. Given a partition $\lambda$, we denote by $\ell_t(\lambda)$ the number of parts of $\lambda$ congruent to $t \pmod r$.  If
$i$ is a part of $\lambda$ with multiplicity $s_i$, we refer to $s_i \pmod r$ as the \textit{residual multiplicity} of $i$ in $\lambda$.  We denote by $\bar{\ell}_t(\lambda)$ the number of different parts in $\lambda$ with residual multiplicity at least $t$. For $1\leq t\leq r-1$, $E_{j,r,t}(n)$ denotes the following difference in the number of parts in the sets of partitions involved in Franklin's identity. $$E_{j,r,t}(n)= \sum_{\lambda\in \mathcal{O}_{j,r}(n)}(\ell_t(\lambda)- \ell_0(\lambda))-\sum_{\lambda \in \D_{j,r}(n)}\overline\ell_t(\lambda).$$ In section \ref{smain} we give analytic and combinatorial proofs of the following theorem. 

\begin{theorem} \label{mainres} For all integers $n,j,r,t$ with $n,j \geq 0$,  $r\geq 2$  and $1\leq t\leq r-1$, we have \begin{equation}\label{beck3}E_{j,r,t}(n)=(j+1)|\mathcal{O}_{j+1,r}(n)|-j|\mathcal{O}_{j,r}(n)|= (j+1)|\D_{j+1,r}(n)|-j|\D_{j,r}(n)|.\end{equation}\end{theorem}

For the particular case $j=0$, we gave analytic and combinatorial proofs of Theorem \ref{mainres} in \cite{BW20b}. For the combinatorial proof, we made use of a recent bijection introduced by Xiong and Keith in \cite{XK19}.  The case $j=0$, $t=1$ is given in \cite{FT17} where it is proved analytically.   We were motivated to discover Theorem \ref{mainres} by the result in \cite{FT17}.

In section \ref{smain} we also show analytically and combinatorially how Theorem \ref{main} follows from Theorem \ref{mainres}. In particular, this leads to a new combinatorial proof of Theorem \ref{main1}. In section \ref{beck2} we introduce a second Beck-type companion identity  for Franklin's identity and give analytic and combinatorial proofs. We note that the analytic proof is also included in \cite{BW20c}.

\section{Proof of Theorems \ref{mainres} and \ref{main}} \label{smain}

\subsection{Analytic proofs} \begin{proof}[Proof of Theorem \ref{mainres}] For the remainder of this section, fix  integers $r\geq 2$ and $1\leq t\leq r-1$. Denote by $\mathcal{O}_{j,r, t}(m,n)$, respectively $\D_{j,r, t}(m,n)$,  the subset of partitions in $\mathcal{O}_{j,r}(n)$, respectively $\D_{j,r}(n)$,  with $m$ parts congruent to $t \pmod r$, respectively with $m$ different parts with residual multiplicity at least $t$. By analogy, we denote by $\mathcal{O}_{j,r, 0}(m,n)$ the subset of partitions in $\mathcal{O}_{j,r}(n)$ with $m$ parts divisible by $r$. %Additionally, denote by $\mathcal{P}_{\equiv 0}(m, n)$ the subset of partitions in $\O_{j,r}(n)$ with $m$ parts congruent to $0 \pmod r$.

We start with the trivariate generating functions for the sequences $\{|\mathcal{O}_{j,r, t}(m,n)|\}$, $\{|\D_{j,r, t}(m,n)|\}$, and $\{|\mathcal{O}_{j, r, 0}(m,n)|\}$. Let $z,w$, and $q$ be complex variables of modulus less than $1$ so that all series converge absolutely. We define $$\bO_{r,t}(z,w,q):=\sum_{n=0}^\infty\sum_{m=0}^\infty\sum_{j=0}^\infty|\mathcal{O}_{j,r, t}(m,n)|z^mw^jq^n,$$ $$\bD_{r,t}(z,w,q):=\sum_{n=0}^\infty\sum_{m=0}^\infty\sum_{j=0}^\infty|\D_{j,r, t}(m,n)|z^mw^jq^n,$$ and $${\bO}_{r,0}(z,w,q):=\sum_{n=0}^\infty\sum_{m=0}^\infty\sum_{j=0}^\infty|\mathcal{O}_{j, r, 0}(m,n)|z^mw^jq^n.$$

We have  $\bO_{r,t}(z,w,q)   =$ \begin{multline*} \prod_{n=1}^\infty(1+wq^{rn}+wq^{2(rn)}+wq^{3(rn)}+\cdots)\\ \cdot \prod_{\substack{n=0}}^\infty\frac{1}{(1-q^{rn+1})(1 - q^{rn + 2})\dotsm(1-q^{rn + t-1})(1 - zq^{rn + t})(1 - q^{rn + t + 1}) \dotsc (1 - q^{rn + r - 1})} \\  =   \prod_{n=1}^\infty\left(1+\frac{wq^{rn}}{1-q^{rn}}\right)  \cdot  \prod_{\substack{n=1\\ n\not\equiv 0\!\!\! \!\! \pmod r}}^\infty\frac{1}{1-q^n}\cdot \prod_{n=0}^\infty\frac{1-q^{rn+t}}{1-zq^{rn+t}} , \end{multline*} $ \bD_{r,t}(z,w,q) =$  \begin{multline*} \prod_{n=1}^\infty (1 + q^n + q^{2n} + \cdots + q^{(t-1)n} + zq^{tn} + zq^{(t+1)n} + \cdots + zq^{(r-1)n})(1+wq^{rn}+ wq^{(2r)n}+ \cdots)\\ 
 =   \prod_{n=1}^\infty\left(1+\frac{wq^{rn}}{1-q^{rn}}\right)\cdot \prod_{n=1}^\infty \left(1 + q^n + q^{2n} + \cdots + q^{(t-1)n} + zq^{tn} + zq^{(t+1)n} + \cdots + zq^{(r-1)n} \right),\end{multline*} and $\bO_{r,0}(z,w,q)=$ \begin{multline*} \prod_{n=1}^\infty(1+wzq^{rn}+wz^2q^{2(rn)}+wz^3q^{3(rn)}+\cdots)\cdot \prod_{\substack{n=1\\ n\not\equiv 0\!\!\! \!\! \pmod r}}^\infty\frac{1}{1-q^n}   \\  =   \prod_{n=1}^\infty\left(1+\frac{wzq^{rn}}{1-zq^{rn}}\right)\cdot \prod_{\substack{n=1\\ n\not\equiv 0\!\!\! \!\! \pmod r}}^\infty\frac{1}{1-q^n}. \end{multline*}  

Clearly, $$ \sum_{n=0}^\infty\sum_{j=0}^\infty E_{j,r,t}(n)w^jq^n  = \left.\frac{\del}{\del z}\right|_{z=1}\bigl(\bO_{r,t}(z,w,q) - \bO_{r,0}(z,w,q)-\bD_{r,t}(z,w,q)\bigr).$$ Using logarithmic differentiation, we obtain \begin{multline*}\left.\frac{\del}{\del z}\right|_{z=1}\bO_{r,t}(z,w,q)=   \prod_{n=1}^\infty\left(1+\frac{wq^{rn}}{1-q^{rn}}\right)\cdot \prod_{\substack{n=1\\ n\not\equiv 0\!\!\! \!\! \pmod r}}^\infty\frac{1}{1-q^n}\cdot \sum_{m=1}^\infty\frac{q^{mr + t}}{1-q^{mr + t}},\end{multline*}
\begin{multline*}\left.\frac{\del}{\del z}\right|_{z=1}\bD_{r,t}(z,w,q)=   \prod_{n=1}^\infty\left(1+\frac{wq^{rn}}{1-q^{rn}}\right)\cdot \prod_{n=1}^\infty\frac{1-q^{rn}}{1-q^n}  \\ \cdot \sum_{m=1}^\infty\frac{q^{tm} + q^{(t+1)m} + \cdots + q^{(r-1)m}}{1 + q^m + q^{2m} + \cdots + q^{(t-1)m} + q^{tm} + q^{(t+1)m} + \cdots + q^{(r-1)m} } \\ = \prod_{n=1}^\infty\left(1+\frac{wq^{rn}}{1-q^{rn}}\right)\cdot \prod_{n=1}^\infty\frac{1-q^{rn}}{1-q^n} \cdot \sum_{m=1}^\infty\frac{q^{tm} - q^{rm}}{1 - q^{rm}},\end{multline*} and \begin{multline*}\left.\frac{\del}{\del z}\right|_{z=1}\bO_{r,0}(z,w,q)=   \prod_{n=1}^\infty\left(1+\frac{wq^{rn}}{1-q^{rn}}\right)\cdot \prod_{\substack{n=1\\ n\not\equiv 0\!\!\! \!\! \pmod r}}^\infty\frac{1}{1-q^n} \\ \cdot \left( \sum_{m=1}^\infty\frac{1 - q^{rm}}{1 - (1 - w)q^{rm}} \cdot \frac{wq^{rm}}{(1-q^{rm})^2}\right) \\ =  \prod_{n=1}^\infty\left(1+\frac{wq^{rn}}{1-q^{rn}}\right)\cdot \prod_{\substack{n=1\\ n\not\equiv 0\!\!\! \!\! \pmod r}}^\infty\frac{1}{1-q^n}\cdot \sum_{m=1}^\infty\frac{wq^{rm}}{(1 - (1-w)q^{rm})(1 - q^{rm})}.\end{multline*}\medskip

We have $$\sum_{n=0}^\infty \frac{q^{nr+t}}{1-q^{nr+t}}= \sum_{n=0}^\infty\sum_{m=1}^\infty q^{m(nr+t)}= \sum_{m=1}^\infty q^{tm}\sum_{n=0}^\infty q^{nrm}=\sum_{m=1}^\infty\frac{q^{tm}}{1-q^{rm}}.$$

It follows that 
\begin{multline*}\sum_{n=0}^\infty\sum_{j=0}^\infty E_{j,r,t}(n) w^jq^n  =  \prod_{n=1}^\infty\left(1+\frac{wq^{rn}}{1-q^{rn}}\right) \prod_{\substack{n=1\\ n\not\equiv 0\!\!\! \!\! \pmod r}}^\infty\frac{1}{1-q^n}\\ \cdot  \sum_{m=1}^\infty\left( \frac{q^{tm}}{1 - q^{rm}} - \frac{q^{tm} - q^{rm}}{1 - q^{rm}} - \frac{wq^{rm}}{(1 - q^{rm})(1 - (1 - w)q^{rm})}\right)\\  =  \prod_{n=1}^\infty\left(1+\frac{wq^{rn}}{1-q^{rn}}\right) \prod_{\substack{n=1\\ n\not\equiv 0\!\!\! \!\! \pmod r}}^\infty\frac{1}{1-q^n} \sum_{m=1}^\infty\left( \frac{q^{rm}}{1 - q^{rm}} - \frac{wq^{rm}}{(1 - q^{rm})(1 - (1 - w)q^{rm})}\right). \end{multline*} 

Since $$\frac{wq^{mr}}{(1-q^{mr})(1-(1-w)q^{mr})}= \frac{1}{1-q^{rm}}-\frac{1}{1-(1-w)q^{rm}},$$  we have $$\sum_{m=1}^\infty\left(\frac{q^{rm}}{1-q^{rm}} - \frac{wq^{mr}}{(1-q^{mr})(1-(1-w)q^{mr})}\right)= \sum_{m=1}^\infty\frac{(1-w)q^{mr}}{1-(1-w)q^{mr}}.$$ Thus, 
\begin{multline*}\sum_{n=0}^\infty\sum_{j=0}^\infty E_{j,r,t}(n)w^jq^n  \\ =  \prod_{n=1}^\infty\left(1+\frac{wq^{rn}}{1-q^{rn}}\right) \prod_{\substack{n=1\\ n\not\equiv 0\!\!\! \!\! \pmod r}}^\infty\frac{1}{1-q^n} \left( \sum_{m=1}^\infty\frac{(1-w)q^{mr}}{1-(1-w)q^{mr}}\right)\\ = \ds (1-w)\sum_{m=1}^\infty\left(\frac{q^{mr}}{1-q^{mr}}\prod_{\substack{n=1\\ n \neq m}}^\infty\left(1+\frac{wq^{rn}}{1-q^{rn}}\right)\cdot \prod_{\substack{n=1\\ n\not\equiv 0 \pmod r}}^\infty\frac{1}{1-q^n}\right)\\ = (1-w)\sum_{n=0}^\infty\sum_{j=0}^\infty (j+1)|\mathcal{O}_{{j+1},r}(n)|w^jq^n\\ = \sum_{n=0}^\infty \sum_{j=0}^\infty ((j+1)|\mathcal{O}_{j+1,r}(n)|-j|\mathcal{O}_{j,r}(n)|)w^jq^n.\end{multline*}

To see the second to last equality above, notice that the exponent of $q$ coming from the term $\ds \frac{q^{mr}}{1-q^{mr}}$ keeps track of the part $mr$ (with multiplicity) of the partition $\lambda$ in $\mathcal{O}_{j+1,r}(n)$ and the exponent of $q$ coming from the first product keeps track of the other $j$ parts of $\lambda$ divisible by $r$ and it is weighted by $w^j$. Since each of the $j+1$ different parts of $\lambda$ divisible by $r$ can be contributed by the first term, $\lambda$ contributes $j+1$ to the coefficient of $w^jq^n$.  
\end{proof}

\begin{proof}[Proof of Theorem \ref{main}] We sum identities \eqref{beck3} for $t=1, 2, \ldots, r-1$. Then \begin{align*}\label{beck3a}\sum_{t=1}^{r-1}E_{j,r,t}(n)& =(r-1)((j+1)|\mathcal{O}_{j+1,r}(n)|-j|\mathcal{O}_{j,r}(n)|)\\ & =(r-1)( (j+1)|\D_{j+1,r}(n)|-j|\D_{j,r}(n)|).\end{align*}

It remains to show that $$\sum_{t=1}^{r-1}E_{j,r,t}(n)=
b_{j,r}(n).$$

We have \begin{align*}\sum_{t=1}^{r-1}E_{j,r,t}(n)& =\sum_{t=1}^{r-1}\left(\sum_{\lambda\in \mathcal{O}_{j,r}(n)}(\ell_t(\lambda)- \ell_0(\lambda))-\sum_{\lambda \in \D_{j,r}(n)}\overline\ell_t(\lambda)\right)\\ & =\sum_{\lambda\in \mathcal{O}_{j,r}(n)}\ell(\lambda)-r \sum_{\lambda\in \mathcal{O}_{j,r}(n)}\ell_0(\lambda)- \sum_{t=1}^{r-1}\sum_{\lambda \in \D_{j,r}(n)}\overline\ell_t(\lambda).\end{align*}

Given a partition $\lambda \in \D_{j,r}(n)$, each part of $\lambda$ is counted in $\ds \sum_{t=1}^{r-1}\sum_{\lambda\in \D_{j,r}(n)}\overline\ell_t(\lambda)$ as many times as its residual multiplicity.

If $s_i=a_i\cdot r+d_i$ is the multiplicity of $i$ in the partition $\lambda$, we refer to $a_i\cdot r$ as the \textit{nonresidual multiplicity} of $i$ in $\lambda$. Denote by $\overline\D_{j,r}(m,n)$ the subset of partitions $\lambda$ in $\D_{j,r}(n)$ such that the sum of the  nonresidual multiplicities of parts in $\lambda$ equals $m$.  Let $$\overline\bD_{r}(z,w,q):=\sum_{n=0}^\infty\sum_{m=0}^\infty\sum_{j=0}^\infty|\overline\D_{j,r}(m,n)|z^mw^jq^n.$$ Then, 
$$ \overline\bD_{r}(z,w,q) =   \prod_{n=1}^\infty\left(1+\frac{wz^rq^{rn}}{1-z^rq^{rn}}\right)\cdot \prod_{n=1}^\infty\frac{1-q^{rn}}{1-q^n}=\bO_{r,0}(z^r,w,q).$$ Thus $$\left.\frac{\del}{\del z}\right|_{z=1}\overline\bD_{r}(z,w,q)=r\left.\frac{\del}{\del z}\right|_{z=1}\bO_{r,0}(z,w,q).$$ Therefore $$\sum_{t=1}^{r-1}E_{j,r,t}(n)=
b_{j,r}(n).$$
\end{proof}

\subsection{Combinatorial proofs}

\begin{proof}[Proof of Theorem \ref{mainres}]

First, we briefly recall Franklin's bijective proof of Theorem \ref{FT} and introduce a modification that serves our purposes. Denote by $\psi:\mathcal{O}_{0,r}(n)\to\D_{0,r}(n)$ the bijection introduced by Xiong and Keith in \cite{XK19}. The reader can also find a succinct description of the bijection $\psi$ in \cite{BW20b}. Then Franklin's bijection $\varphi: \mathcal{O}_{j,r}(n)\to \D_{j,r}(n)$ is defined as follows. Let $\lambda \in \mathcal{O}_{j,r}(n)$.  Suppose the parts of $\lambda$ divisible by $r$  are $(m_ir)^{k_i}$ with $m_i, k_i > 0$ for $1\leq i\leq j$, and   $m_i$  are distinct. 
Let $\overline \lambda=\lambda \setminus \cup_{i=1}^j(m_ir)^{k_i}$ be the partition obtained from $\lambda$ by removing  all  parts equal to  $m_ir$ for $1\leq i\leq j$. Then $\overline\lambda \in \mathcal{O}_{0,r}(n - \sum_{i=1}^j k_im_i r)$. %We know it will have no parts divisible by $r$ because we have just removed all parts divisible by $r$ and  subtracting $rk$ from some part not divisible by $r$ will result in a new part not divisible by $r$. 
Let $\overline \mu=\psi(\overline \lambda)\in \D_{0,r}(n - \sum_{i=1}^j k_im_i r)$ be the image of $\overline \lambda$ under the Xiong-Keith bijection.  Finally, let $\mu=\overline \mu\cup ((m_1)^{k_1r},(m_2)^{k_2r},\ldots, (m_j)^{k_jr})$. Since the parts $m_i$, $1\leq i\leq j,$ are all distinct and they are the only parts repeated at least $r$ times, we have that  $\mu \in  \D_{j,r}(n)$.   Set $\varphi(\lambda)=\mu$.  

To describe the inverse mapping, let $\mu \in \D_{j,r}(n)$. Suppose the $j$ different parts that are repeated at least $r$ times are $m_i$, $1\leq i\leq j,$ and each $m_i$ has  multiplicity $a_i$ in $\mu$. For each $1\leq i \leq j$, write  $a_i=k_i r+d_i$ with $0\leq d_i\leq r-1$, and remove $k_ir$ parts equal to $m_i$ from $\mu$ to obtain a partition $\overline \mu \in  \D_{0,r}(n - \sum_{i=1}^j k_im_i r)$. Let $\overline \lambda =\psi^{-1}(\overline \mu) \in \mathcal{O}_{0,r}(n - \sum_{i=1}^j k_im_i r)$ be the image of $\overline \mu$ under the inverse of the Xiong-Keith bijection. Let $\lambda=\overline \lambda \cup ((m_1r)^{k_1},(m_2r)^{k_2},\ldots, (m_jr)^{k_j})$. Clearly, $\lambda \in \mathcal{O}_{j,r}(n)$. Then $\varphi^{-1}(\mu)=\lambda$.

In \cite{BW20b}, we  provided a combinatorial proof of \eqref{beck3} in the case $j = 0$, i.e. for the identity  

\begin{equation}\label{beck3-0}E_{0,r,t}(n)=|\mathcal{O}_{1,r}(n)|= |\D_{1,r}(n)|.\end{equation} %The difference in the number of parts in all partitions in $\O_{0,r}(n)$ and the number of parts in all partitions in $\mathcal{D}_{0,r}(n)$ equals $(r - 1)|\O_{1,r}(n)|$. 
We will make use of the combinatorial proof of \eqref{beck3-0}  to prove combinatorially the case $j > 0$ of \eqref{beck3}.  Because of Franklin's bijection we only need to prove the identity involving $|\mathcal{O}_{j,r}(n)|$. 

 Using Franklin's  bijection $\varphi:\mathcal{O}_{j,r}(n)\to \D_{j,r}(n)$ described above, we have 
\begin{align}\nonumber E_{j,r,t}(n)& = \sum_{\lambda\in \mathcal{O}_{j,r}(n)}(\ell_t(\lambda)- \ell_0(\lambda))-\sum_{\lambda \in \D_{j,r}(n)}\overline\ell_t(\lambda)\\ \nonumber & =  \sum_{\lambda \in \mathcal{O}_{j,r}(n)}(\ell_t(\lambda)-\ell_0(\lambda)-\overline\ell_t(\varphi(\lambda)).\end{align}

Let $\mathbf m=(m_1, m_2, \ldots, m_j)$ with $m_i$ distinct, and $\mathbf k=(k_1, k_2, \ldots, k_j)$ be two $j$-tuples of positive integers, and denote by $\mathcal{O}_{j, r}^{\mathbf m, \mathbf k} (n)$ the set of partitions $\lambda$  in $\mathcal{O}_{j, r}(n)$ in which the  $j$ different parts divisible by $r$ are precisely $(m_ir)^{k_i}$ with $m_i, k_i>0$, for $i = 1, 2, \dotsc, j$, and $m_i$ distinct.
Given  $\lambda$  in $\mathcal{O}_{j, r}^{\mathbf m, \mathbf k} (n)$, as in the proof of Franklin's identity,  we write $\ds\lambda=\overline \lambda \ds\cup_{i=1}^j (m_ir)^{k_i}$ with $\overline{\lambda}\in \mathcal{O}_{0,r}(n - \sum\limits_{i=1}^{j} rm_ik_i)$. Then, 
\begin{align*} \ell_t(\lambda)&=\ell_t(\overline{\lambda})\\  \ell_0(\lambda)&= \sum\limits_{i=1}^{j} k_i\\  \sum_{\lambda \in \mathcal{O}_{j,r}^{\mathbf m, \mathbf k}(n)}\overline\ell_t(\varphi(\lambda))&= \sum_{\overline\mu \in \D_{0,r}(n- \sum\limits_{i=1}^{j} rm_ik_i)} \overline\ell_t(\overline\mu).\end{align*} Then  $$E_{j,r,t}(n)= \sum_{(\mathbf m, \mathbf k)} E_{0,r,t}(n-r \mathbf m \cdot \mathbf k)- \sum_{\lambda\in \mathcal{O}_{j,r}(n)} \ell_0(\lambda),$$ where the first sum is over all pairs $(\mathbf m, \mathbf k)$ of $j$-tuples of positive integers  such that $\mathbf m$ has distinct components. Moreover,  $\mathbf m \cdot \mathbf k$ represents the dot product of the vectors $\mathbf m$ and $\mathbf k$.

Using the combinatorial proof of \eqref{beck3-0} we have \begin{equation}\label{diff3}E_{j,r,t}(n)= \sum_{(\mathbf m, \mathbf k)} |\mathcal{O}_{1,r}(n-r \mathbf m \cdot \mathbf k)|- \sum_{\lambda\in \mathcal{O}_{j,r}(n)} \ell_0(\lambda).\end{equation}

Next, for fixed $j$-tuples of positive integers $\mathbf m=(m_1, m_2, \ldots, m_j)$ with $m_i$ distinct, and $\mathbf k=(k_1, k_2, \ldots, k_j)$, we reinterpret $|\mathcal{O}_{1,r}(n-r \mathbf m \cdot \mathbf k)|$.
Consider the mapping $\zeta: \mu\mapsto \mu\cup_{i=1}^j(m_ir)^{k_i}$ on $\mathcal{O}_{1,r}(n-r \mathbf m \cdot \mathbf k)$. 
 If $\mu \in \mathcal{O}_{1,r}(n-r \mathbf m \cdot \mathbf k)$, then $\zeta(\mu)$ fits in exactly one of the following two cases.  

(i)  If $m_tr$ is a part of $\mu$ for some $1\leq t\leq j$, then $\zeta(\mu) \in  \mathcal{O}_{j, r}(n)$ and it contains each part $m_ir$ with multiplicity $k_i$ if $i\neq t$, and part $m_tr$ with multiplicity larger than $k_t$, or

(ii) If $\mu$ does not contain any $m_ir$, $1\leq i\leq j$, as a part, then $\zeta(\mu) \in  \mathcal{O}_{j + 1, r}(n)$ and it contains each part $m_ir$, $1\leq i\leq j$  with multiplicity $k_i$.

We obtain a bijection between $\mathcal{O}_{1,r}(n-r \mathbf m \cdot \mathbf k)$ and the union of the subsets of $ \mathcal{O}_{j, r}(n)$, respectively  $\mathcal{O}_{j+1, r}(n)$, described in (i), respectively (ii) above.

Now let $\eta \in \mathcal{O}_{j, r}(n)$. Suppose $(m_ir)^{k_i}$, $1\leq i\leq j$, $m_i$ distinct,  are the parts divisible by $r$ in $\eta$. %and let $m(\eta)$ be the total number of parts divisible by $r$ in $\eta$, i.e., $m(\eta)=\sum_{i=1}^jk_i$. 
The contribution of $\eta$ to $\sum_{(\mathbf m, \mathbf k)} |\mathcal{O}_{1,r}(n-r \mathbf m \cdot \mathbf k)|$ comes from (i) above.   For each choice of $t$,  \ $1 \leq t \leq j$,  and  \ $h_t, \ 1 \leq h_t \leq k_t - 1$, the partition $\eta$ can be written as 

$$\eta = \mu \cup_{i\neq t}(m_i r)^{k_i} \cup ((m_t r)^{h_t}),$$ with $\mu \in \mathcal{O}_{1,r}(n-rm_th_t-\sum\limits_{i\neq t}rm_ik_i)$ such that  $m_tr$ is the only part of $\mu$ divisible by $r$ and it has multiplicity $k_t-h_t\geq 1$ in $\mu$.

Thus,  $\eta$ contributes  $\sum\limits_{i=1}^{j} (k_i - 1)=\ell_0(\eta) - j$  to $\sum_{(\mathbf m, \mathbf k)} |\mathcal{O}_{1,r}(n-r \mathbf m \cdot \mathbf k)|$.

Next, let $\eta \in \mathcal{O}_{j + 1, r}(n)$. Suppose $(m_ir)^{k_i}$, $1\leq i\leq j+1$, $m_i$ distinct,  are the parts divisible by $r$ in $\eta$. The contribution of $\eta$ to $\sum_{(\mathbf m, \mathbf k)} |\mathcal{O}_{1,r}(n-r \mathbf m \cdot \mathbf k)|$ comes from (ii) above. For each choice of $t$,  \ $1 \leq t \leq j+1$, the partition $\eta$ can be written as 
$$\eta = \mu \cup_{i\neq t}(m_i r)^{k_i},$$ with $\mu \in \mathcal{O}_{1,r}(n-\sum\limits_{i\neq t}rm_ik_i)$ such that $m_tr$ is  the only  part of $\mu$ divisible by $r$ and it has multiplicity $k_t$ in $\mu$.  Thus,  $\eta$ contributes  $(j + 1)$ to $\sum_{(\mathbf m, \mathbf k)} |\mathcal{O}_{1,r}(n-r \mathbf m \cdot \mathbf k)|$.

In total, we have
  \begin{multline*}
\sum_{(\mathbf m, \mathbf k)} |\mathcal{O}_{1,r}(n-r \mathbf m \cdot \mathbf k)| =  
\\    \sum\limits_{\eta \in \mathcal{O}_{j, r}(n)} (\ell_0(\eta) - j)
+ \sum\limits_{\eta \in \mathcal{O}_{j + 1, r}(n)} (j + 1)
= \\ -j |\mathcal{O}_{j, r}(n)| + (j + 1) |\mathcal{O}_{j + 1, r} (n)|+   \sum\limits_{\eta \in \mathcal{O}_{j, r}(n)} \ell_0(\eta) .
\end{multline*}
Using identity \eqref{diff3}, this finishes the combinatorial proof of Theorem \ref{mainres}. \end{proof}

\begin{proof}[Proof of Theorem \ref{main}] As in the analytic proof, we need to show that $r$ times the number of parts divisible by $r$ in all partitions in $\mathcal{O}_{j,r}(n)$ equals the sum of the nonresidual multiplicities of parts in all partitions in $\D_{j,r}(n)$. This follows immediately from Franklin's bijective proof of Theorem \ref{FT}.

\end{proof}

\section{A second Beck-type identity}\label{beck2}

Let ${T}_{j, r} (n)$ denote the number of (different) parts with multiplicity between $r + 1$ and $2r - 1$ in all partitions in $\D_{j, r}(n)$. Let $b'_{j,r}(n)$ be the difference in the number of different parts in $\D_{j, r}(n)$ and the number of different parts in $\mathcal{O}_{j, r}(n)$. If we denote by $\bar{\ell}(\lambda)$ the number of different parts in $\lambda$, then $$b'_{j,r}(n)=\sum_{\lambda \in \D_{j, r}(n)} \bar{\ell}(\lambda)-\sum_{\lambda \in \mathcal{O}_{j, r}(n)} \bar{\ell}(\lambda).$$ In \cite{Y18}, Yang proved analytically and combinatorially that \begin{equation} \label{b20} b'_{0,r}(n)={T}_{1,r}(n).\end{equation} We note that in \cite{Y18}, ${T}_{1,r}(n)$ is viewed as the cardinality of the subset  of  $\D_{1, r}(n)$ consisting of partitions in which one part is repeated between $r + 1$ and $2r - 1$ times and all other parts appear less than $r$ times. We denote this subset by $\T(n)$. In \cite{BW20b}, we gave a new combinatorial proof of \eqref{b20}. 

Identity \eqref{b20} generalizes to a companion identity for Franklin's identity as follows. Denote by $b'_{\leq j,r}(n)$  the difference in the number of different parts in all partitions in $\D_{\leq j,r}(n)$ and   the number of different parts in all partitions in $\mathcal{O}_{\leq j,r}(n)$.  

\begin{theorem} \label{C1} Let  $n,j, r$ be non-negative integers with  $r\geq 2$. Then, $$b'_{\leq j,r}(n)={T}_{j+1, r}(n).$$
\end{theorem}

To prove Theorem \ref {C1}, we  apply the next theorem repeatedly. 

\begin{theorem} \label{b2} For all non-negative integers $n,j,r$ with $r\geq 2$, we have
$$b'_{j,r}(n)={T}_{j+1, r}(n) - {T}_{j, r}(n).$$
\end{theorem}

\begin{proof}[Analytic Proof] The argument of this proof is similar to that of Theorem \ref{main} and we omit many of the details. 
 Denote by $\mathcal{O}'_{j,r}(m,n)$, respectively $\D'_{j,r}(m,n)$,  the subset of partitions in $\mathcal{O}_{j,r}(n)$, respectively $\D_{j,r}(n)$,  with $m$ distinct parts. The trivariate generating functions for the sequences $\{|\mathcal{O}'_{j,r}(m,n)|\}$ and $\{|\D'_{j,r}(m,n)|\}$ are respectively 
 \begin{multline*}\bO'_r(z,w,q):=\prod_{n=1}^\infty(1+wzq^{rn}+wzq^{2(rn)}+wzq^{3(rn)}+\cdots)\cdot \prod_{\substack{n=1\\ n\not\equiv 0\!\!\! \!\! \pmod r}}^\infty\left(1 + \frac{zq^n}{1-q^n}\right) \\  =   \prod_{n=1}^\infty\left(1+\frac{wzq^{rn}}{1-q^{rn}}\right)\cdot \prod_{\substack{n=1\\ n\not\equiv 0\!\!\! \!\! \pmod r}}^\infty\left(1 + \frac{zq^n}{1-q^n}\right), \end{multline*}
and   \begin{multline*}\bD'_r(z,w,q):= \prod_{n=1}^\infty(1+zq^{n}+zq^{2n}+\cdots +zq^{(r-1)n}+wzq^{rn}+ wzq^{(r+1)n}+ \cdots) \\  = \prod_{n=1}^\infty(1 - z + (z + zq^{n}+zq^{2n}+\cdots +zq^{(r-1)n})(1+wq^{rn}+ wq^{(2r)n}+ \cdots))\\ 
 =   \prod_{n=1}^\infty \left( 1 - z + z \frac{1-q^{rn}}{1-q^n}\left(1+\frac{wq^{rn}}{1-q^{rn}}\right)\right).\end{multline*}
 
  As in the proof of Theorem \ref{main},
  \begin{multline*}\ds \sum_{n=0}^\infty\sum_{j=0}^\infty b'_{j,r}(n)w^jq^n  = \left.\frac{\del}{\del z}\right|_{z=1}\bigl(\bD'_r(z,w,q)-\bO'_r(z,w,q)\bigr)=\\
 \bD'_r(1,w,q) \cdot \left(\sum_{m=1}^\infty \left( 1 - \frac{1 - q^m}{1 - (1-w)q^{rm}} \right) - \sum_{m=1}^\infty \frac{wq^{rm}}{1- (1-w)q^{rm}} - \sum_{\substack{m=1\\ m\not\equiv 0\!\!\! \!\! \pmod r}}^\infty q^m \right)= \end{multline*}   \begin{multline*}
  (1 - w) \sum_{m = 1}^\infty \left( \frac{q^{(r + 1)m} - q^{2rm}}{1 - q^{m}}  \prod_{\substack{n=1\\ n \neq m}}^\infty\frac{1 - ( 1- w)q^{rn}}{1-q^{rn}}\cdot  \prod_{\substack{n=1\\ n \neq m}}^\infty \frac{1-q^{rn}}{1-q^n} \right)  = \\ (1 - w)\sum_{m = 1}^\infty \left( (q^{(r + 1)m} + q^{(r + 2)m} + \dotsm + q^{(2r - 1)m}) \cdot    \prod_{\substack{n=1\\ n \neq m}}^\infty\left(1+\frac{wq^{rn}}{1-q^{rn}}\right)\cdot  \prod_{\substack{n=1\\ n \neq m}}^\infty \frac{1-q^{rn}}{1-q^n} \right)= \\ (1 - w) \sum_{n=0}^\infty\sum_{j=0}^\infty {T}_{j+1, r}(n) w^jq^n=   \sum_{n=0}^\infty\sum_{j=0}^\infty ({T}_{j+1, r}(n) - {T}_{j, r}(n)) w^jq^n .\end{multline*}
\end{proof}

\begin{proof}[Combinatorial Proof] 
To count the difference in the number of distinct parts in all partitions in $\D_{j , r}(n)$ and the number of distinct parts in all partitions in $\mathcal{O}_{j, r}(n)$, we use the inverse of Franklin's  bijection $\varphi^{-1}:\D_{j,r}(n)\to \mathcal{O}_{j,r}(n)$. We have

\begin{equation}\label{diff}b'_{j,r}(n)= \sum_{\lambda \in \D_{j,r}(n)}\bar{\ell}(\lambda)- \sum_{\lambda \in \mathcal{O}_{j,r}(n)}\bar{\ell}(\lambda)= \sum_{\lambda \in \D_{j,r}(n)}(\bar{\ell}(\lambda)-\bar{\ell}(\varphi^{-1}(\lambda)).\end{equation}

Let $\mathbf m=(m_1, m_2, \ldots, m_j)$ with $m_i$ distinct, and $\mathbf k=(k_1, k_2, \ldots, k_j)$ be two $j$-tuples of positive integers, and denote by $\D_{j, r}^{\mathbf m, \mathbf k} (n)$ the set of partitions $\lambda$  in $\D_{j, r}(n)$ such that  the  $j$ different parts of $\lambda$ repeated at least  $r$ times are precisely $m_i$,  $i = 1, 2, \dotsc, j$, and the nonresidual multiplicity of each $m_i$ is $rk_i$. We also denote by $S_{j,r}^{\mathbf m, \mathbf k}(n)$ the number of different parts in all partitions  in $\D_{j, r}^{\mathbf m, \mathbf k} (n)$ that appear  in a partition with multiplicity at least $r$ but not divisible by $r$.

Let $\lambda$ be a partition in $\D_{j, r}^{\mathbf m, \mathbf k} (n)$.  When $\ds\lambda=\overline \lambda \ds\cup_{i=1}^j (m_i)^{rk_i}$ is mapped to $\ds\varphi^{-1} (\lambda)=\psi^{-1}( \overline\lambda)\ds\cup_{i=1}^j (m_ir)^{k_i}$, the contribution of $\lambda$ to $\bar{\ell}(\lambda)-\bar{\ell}(\varphi^{-1}(\lambda))$ is as follows: \medskip

 $0$  from mapping $(m_i)^{rk_i}$ to $(m_ir)^{k_i},$ $i = 1, 2, \dotsc , j$, and \medskip

  $\bar{\ell}(\overline{\lambda}) - \bar{\ell}(\psi^{-1}(\overline{\lambda}))-s(\overline{\lambda})$, where $s(\overline{\lambda})$ is the number of different parts of $\overline{\lambda}$ equal to $m_i$,  $i = 1, 2, \dotsc , j$.
  
  Then, \begin{align*} b'_{j,r}(n)& =\sum_{({\mathbf m, \mathbf k} )}\left(\sum_{\lambda \in \D_{j,r}^{\mathbf m, \mathbf k} (n)}(\bar{\ell}(\overline\lambda)-\bar{\ell}(\psi^{-1}(\overline\lambda))-S_{j,r}^{\mathbf m, \mathbf k}(n)\right)\\ & =\sum_{({\mathbf m, \mathbf k} )} (b'_{0,r}(n-r\mathbf m\cdot \mathbf k)-S_{j,r}^{\mathbf m, \mathbf k}(n)).\end{align*}

Using the combinatorial proof of \eqref{b20}, we have

  \begin{equation}\label{b'}
 b'_{j,r}(n)=\sum_{({\mathbf m, \mathbf k} )}(|\mathcal{T}_r(n - r\mathbf m\cdot \mathbf k)| -S_{j,r}^{\mathbf m, \mathbf k}(n)).
\end{equation}

Next, for fixed $j$-tuples of positive intergers $\mathbf m=(m_1, m_2, \ldots, m_j)$ with $m_i$ distinct, and $\mathbf k=(k_1, k_2, \ldots, k_j)$, we reinterpret $\mathcal{T}_{r}(n -  r\mathbf m\cdot \mathbf k)$.

Consider $\zeta: \mu\mapsto \mu\cup_{i=1}^j(m_i)^{rk_i}$ as a mapping  on $\mathcal{T}_{r}(n - r\mathbf m\cdot \mathbf k)$.
If $\mu \in \mathcal{T}_{r}(n - r\mathbf m\cdot \mathbf k)$, then $\mu$ contains exactly one part, $\alpha$, that is repeated more than $r$ but less than $2r$ times, and all other parts in $\mu$ are repeated less than $r$ times. So $\zeta(\mu)$ fits in exactly one of the following two cases.  

(i)  If $\alpha = m_t$ for some $1\leq t\leq j$, then $\zeta(\mu) \in  \D_{j, r}(n)$ and it contains each part $m_i$ with nonresidual multiplicity equal to $rk_i$ if $i\neq t$, and part $\alpha=m_t$ with multiplicity larger than $r+rk_t$ but less than $2r+rk_t$ (i.e., the nonresidual multiplicity of $\alpha$ is $(k_t+1)r$ and its multiplicity is not divisible by $r$), or

(ii) If $\alpha$ does not equal  $m_i$ for any $1\leq i\leq j$, then $\zeta(\mu) \in  \D_{j + 1, r}(n)$ and it contains each part $m_i$, $1\leq i\leq j$  with nonresidual multiplicity equal to $rk_i$, and part $\alpha$ with multiplicity larger than $r$ but less than $2r$.

We obtain a bijection between $\mathcal{T}_{r}(n - r\mathbf m\cdot \mathbf k)$ and the union of the subsets of $ \D_{j, r}(n)$, respectively  $\D_{j+1, r}(n)$, described in (i), respectively (ii), above. 

Notice that in (i), the multiplicity of part $\alpha$ is greater than $2r$ and not divisible by $r$.

Then, each partition $\eta\in  \D_{j, r}(n)$ contributes to $\sum_{({\mathbf m, \mathbf k} )}|\mathcal{T}_r(n - r\mathbf m\cdot \mathbf k)|$ the number of parts in $\eta$
 that have multiplicity greater than $2r$ and not divisible by $r$ and  each partition $\nu \in \D_{j + 1, r}(n)$ contributes to $\sum_{({\mathbf m, \mathbf k} )}|\mathcal{T}_r(n - r\mathbf m\cdot \mathbf k)|$ the number of parts in $\nu$ with multiplicity larger than $r$ but less than $2r$. Thus $\sum_{({\mathbf m, \mathbf k} )}|\mathcal{T}_r(n - r\mathbf m\cdot \mathbf k)|$ equals ${T}_{j+1,r}(n)$ plus the total number of parts with multiplicity larger than $2r$ and not divisible by $r$ in all partitions in  $\D_{j, r}(n)$.

%To obtain the contribution of all partitions $\lambda \in \D_{j,r}(n)$ to \eqref{diff}, we consider this process for all possible $j$-tuples of positive integers $\mathbf m=(m_1, m_2, \ldots, m_j)$ and $\mathbf k=(k_1, k_2, \ldots, k_j)$.

%Now let $\eta \in \D_{j, r}(n)$. Suppose $m_i$, $1\leq i\leq j$,  are the distinct parts with multiplicity greater than $r$ in $\eta$. Let $\mathcal{s}'_{2r}(\eta)$ be the number of parts in $\eta$
% that have multiplicity greater than $2r$ and not divisible by $r$. 
%Then $\eta$ contributes $\mathcal{s}'_{2r}(n)$ to $\sum_{({\mathbf m, \mathbf k} )}|\mathcal{T}_r(n - r\mathbf m\cdot \mathbf k)|$.

%Next, let  $\eta \in \D_{j + 1, r}(n)$. Suppose $m_i$, $1\leq i\leq j+1$ are the parts with multiplicity greater than $r$ in $\eta$. The partition $\eta$ contributes $\mathcal{T}_{j + 1, r}$, the number of parts in all partitions in $D_{j + 1, r}(n)$ with multiplicity more than $r$ but less than $2r$.

Next, we observe that   $\sum_{({\mathbf m, \mathbf k} )}S_{j,r}^{\mathbf m, \mathbf k}(n)$ is the total number of parts in all partitions in $D_{j,r}(n)$ that have multiplicity larger than $r$ and not divisible by $r$. 

It follows from \eqref{b'} that
\begin{align*}
b'_{j,r}(n)  = {T}_{j + 1, r}(n) - {T}_{j, r}(n).
\end{align*}

\end{proof}

\section{Concluding Remarks} \label{fg}

%In \cite{BW20a} we showed that analogous identities to \eqref{bg1} and  \eqref{b20} hold for all Euler pairs of order $r$. 

Let $S_1$ and $S_2$ be subsets of the positive integers. We define 
$\widetilde{\mathcal{O}}_{j,r}(n)$ to be  the set of partitions of $n$ with exactly $j$ different parts from $rS_1$ and all other parts from $S_2$ and  $\widetilde\D_{j,r}(n)$ to be the set of partitions of $n$ with  parts in $S_1$ and exactly $j$ different parts repeated at least $r$ times.  Subbarao \cite{S71} proved the following theorem. 
\begin{theorem} \label{sub-thm} Let $r\geq 2$. Then, $|\widetilde{\mathcal{O}}_{0,r}(n)|=|\widetilde\D_{0,r}(n)|$ for all non-negative integers $n$ if and only if $rS_1\subseteq S_1$ and $S_2=S_1\setminus rS_1$. \end{theorem} Andrews \cite{A69}  first discovered this result for $r=2$ and called a pair $(S_1,S_2)$ such that $|\widetilde{\mathcal{O}}_{0,2}(n)|=|\widetilde\D_{0,2}(n)|$  an \textit{Euler pair} since the pair $S_1=\N$ and $S_2=2\N-1$ gives  Euler's identity. By analogy, Subbarao called a pair $(S_1,S_2)$ such that $|\widetilde{\mathcal{O}}_{0,r}(n)|=|\widetilde\D_{0,r}(n)|$  an \textit{Euler pair of order $r$}.  %For examples of Euler pairs of order $r$, we refer the reader to \cite{S71} or \cite{BW20a}.  
In \cite{BW20a},  we showed that, if $(S_1,S_2)$ is an Euler pair,  the identity of Theorem \ref{sub-thm} has companion Beck-type identities analogous to \eqref{bg1} and \eqref{b20}. 

 It is straight forward to show that if $(S_1,S_2)$ is an Euler pair of order $r$, then $|\widetilde{\mathcal{O}}_{j,r}(n)|=|\widetilde\D_{j,r}(n)|$. In \cite{BW20c}, we remark that a similar argument to \cite{BW20a} establishes analogues of Theorems \ref{main} and \ref{b2} and thus analogues of Theorems \ref{main1} and \ref{C1} for all Euler pairs of order $r$. Denote by  $\widetilde b_{j,r}(n)$ the difference in the number of parts in $\widetilde{\mathcal{O}}_{j,r}(n)$ and  the number of parts in $\widetilde\D_{j,r}(n)$. Denote  by $\widetilde b'_{j,r}(n)$ the  analogous difference of the number of different parts. Define $\widetilde b_{\leq j,r}(n)$  and $\widetilde b'_{\leq j,r}(n)$ in analogy to  $b_{\leq j,r}(n)$  and $b'_{\leq j,r}(n)$. Let $\widetilde{{T}}_{j, r} (n)$ denote the number of parts with multiplicity between $r + 1$ and $2r - 1$ in all partitions in $\widetilde D_{j, r}(n)$.   

\begin{theorem}  \label{last} For all non-negative integers $n,j$ and all integers  $r\geq 2$, we have \begin{enumerate}

\item $\ds \frac{1}{r-1} \widetilde b_{\leq j,r}(n)=(j+1)| \widetilde{\mathcal{O}}_{j+1,r}(n)|= (j+1)|\widetilde \D_{j+1,r}(n)|,$

\item $\ds \frac{1}{r-1}\widetilde b_{j,r}(n) =(j+1)|\widetilde{\mathcal{O}}_{j+1,r}(n)|-j|\widetilde{\mathcal{O}}_{j,r}(n)|= (j+1)|\widetilde\D_{j+1,r}(n)|-j|\widetilde\D_{j,r}(n)|$,

\item $\widetilde b'_{\leq j,r}(n)=\widetilde{{T}}_{j+1, r}(n),$

\item $\widetilde b'_{j,r}(n)=\widetilde{{T}}_{j+1, r}(n) - \widetilde{{T}}_{j, r}(n)$.

 \end{enumerate}
\end{theorem} 

Since there are infinite families of Euler pairs of order $r$, we obtain infinite families of new Beck-type identities. 

It is natural to ask if Theorem \ref{mainres} has an analog for Euler pairs or order $r$. It is not clear to us what such an analog would be in general. The main difficulty arises from the fact that $S_2$ can contain multiples of $r$ (though not multiples of $r^2$). It is also not clear what should replace the division algorithm in $S_1$. We leave it as an open question whether a meaningful analogue of Theorem \ref{mainres} for Euler pairs of order $r$ exists.

\bigskip

%\noindent\textit{Department of Mathematics,
%University of Craiova, 200585 Craiova, Romania\\
%mircea.merca@profinfo.edu.ro}

\end{document}